\documentclass[11pt,a4paper]{article}

\usepackage{amsmath,amssymb}
\usepackage{mathrsfs,amsfonts}
\usepackage{bm}
\usepackage{graphicx}
\usepackage{epsfig}
\usepackage{indentfirst}
\usepackage{extarrows}
\usepackage{amsthm}

\begin{document}

\title{\textbf{On-diagonal lower estimate of heat kernel on graphs}}
\author{Yong Lin \quad Yiting Wu }

\date{}
\maketitle

\vspace{-12pt}

\begin{minipage}{145mm}
\noindent{\small\textbf{Abstract} \; The purpose of this paper is to establish a new
continuous-time on-diagonal lower
estimate of heat kernel for large time on graphs. To achieve the
goal, we first give an upper bound of heat kernel in natural graph
metric, and then use this bound and the volume growth condition to 
show the validity of the on-diagonal lower bound.
}
\end{minipage}

\vspace{12pt}

\section{Introduction}

Over the last decades, there has been remarkable progress in our understanding of global pointwise upper and lower bounds
of heat kernel on Riemannian manifolds \cite{Grigoryan2,Davies1,Grigoryan,LY}. The celebrated Li-Yau inequality \cite{LY}
can show that the
heat kernel on
non-negatively curved manifolds satisfies the Gaussian type bounds, that is,
\begin{equation*}
\tag{1.1}
\frac{C_l}{V(x,\sqrt{t})}\exp\left(-c_l\frac{d(x,y)^2}{t}\right)\leq p(t,x,y) \leq
\frac{C_r}{V(x,\sqrt{t})}\exp\left(-c_r\frac{d(x,y)^2}{t}\right).
\end{equation*}

Similar methods have been used to study heat kernel on graphs. Recently, Bauer et al. \cite{BHLLMY} established
a discrete analogue of the Li-Yau inequality and derived a heat kernel estimate under the condition of $CDE(n,0)$.
Despite the upper bound in their results is formulated with Gaussian form, the lower bound is not
quite Gaussian form and is dependent on the parameter $n$. Based on this, Horn et al. \cite{HLLY} improved some results in \cite{BHLLMY} and got the Gaussian type bounds via introducing $CDE'(n,0)$. In addition,
Lin et al. \cite{LLY2,LLY1} investigated the gradient estimate for positive functions and illustrated the
applications of these results in establishing certain upper bounds and lower bounds of the heat kernel on graphs.

In \cite{Davies2}, Davies obtained non-Gaussian upper bounds of heat kernel on graphs in the continuous time setting. The
most interesting feature of his results is that they involve certain functions defined as Legendre transform. In
\cite{bobo}, Bauer et al. established a sharp version of DGG Lemma on graphs. As a direct application, it yields the
Davies's heat kernel estimate. In particular, for large time the DGG Lemma meets the Gaussian type estimate for the heat
kernel in form of $\exp\left(-d^2/2t\right)$. By using the Davies's result, Folz \cite{Folz} established a Long-range weak Gaussian upper bound for heat kernel in the new metric $d_\theta(x,y)$ which was initiated by Davies in \cite{Davies3}.

Comparing to heat kernel upper bound, it is more difficult to get a heat kernel lower bound. Various techniques for
obtaining on-diagonal heat kernel lower bound were discussed earlier, especially, the form
\begin{equation*}
\tag{1.2}
p(t,x,x)\geq \frac{c_1}{V(x,c_2\sqrt{t})},
\end{equation*}
which is valid for all $x\in V$ and some positive constants $c_1,c_2$, has attracted considerable
attentions and interests of researchers.

In \cite{Delmotte}, Delmotte proved that the particular on-diagonal lower bound (1.2) is true on the graphs, which satisfy the continuous-time parabolic Harnack inequality $\mathcal{H}(C_{\mathcal{H}})$. In \cite{HLLY}, Horn et al. derived (1.2) under the condition of $CDE'(n,0)$. Motivated by the idea of Coulhon and Grigoryan \cite{Grigoryan2},
in this paper, we will only use the volume growth condition to obtain a weaker on-diagonal lower estimate of heat kernel
on graphs for large time. In order to achieve this, we first establish an upper bound of heat kernel, which is similar to
Folz's in \cite{Folz}.

It should be noted that, although the upper bound obtained in this paper is
similar to Folz's result, the metric in our result is natural
graph metric, namely, $d(x,y)$ is equal to the number of edges in
the shortest path between $x$ and $y$. And more significantly,
based on this upper bound, it enables us to establish the on-diagonal lower estimate of heat kernel on
graphs for large time under the condition of volume growth.

Before stating the results, we will introduce some definitions,
notations and lemmas in Section 2. We establish our main
results in Sections 3 and 4.

\section{Preliminaries}
In this section, we introduce some definitions, notations and lemmas which
will be used throughout the paper. For more details on these
terminology, we refer the readers to \cite{Davies2,Delmotte,HLLY,AW,RW}.

Suppose that $G=(V,E)$ is a finite or locally finite connected graph, where $V$ denotes the vertex set and $E$ denotes the edge set. We write $y\sim x$ if $y$ is adjacent to $x$, or equivalently $\overline{xy}\in E$, allow the edges on the graph to be weighted. Weights are given by a function $\omega: V \times V\rightarrow[0,\infty)$, that is, the edge $\overline{xy}$ has weight $\omega_{xy}\geq0$ and $\omega_{xy}=\omega_{yx}$.
Furthermore, let $\mu: V\rightarrow \mathbb{R}^+$ be a positive finite measure on the vertices of the $G$. In this paper,
all the graphs in our concern are assumed to satisfy
$$D_\mu:=\max_{x\in V}\frac{m(x)}{\mu(x)}<\infty,$$
where $m(x):=\sum_{y\sim x}\omega_{xy}.$

Let $C(V)$ be the set of real functions on $V$. For any $1\leq p<\infty$, we denote  by
$$\ell^p(V,\mu)=\left \{f\in C(V):\sum_{x \in V} \mu(x)|f(x)|^p<\infty\right \}$$
the set of $\ell^p$ integrable functions on $V$ with respect to
the measure $\mu$.  For $p=\infty$, let
$$\ell^{\infty}(V,\mu)=\left\{f\in C(V):\sup_{x\in V}|f(x)|<\infty\right\}.$$

The standard inner product is defined by
$$\big<f,g\big>=\sum_{x\in V}\mu(x)f(x)g(x),\;\;\mathrm{for}\;\mathrm{all}\; f,g\in \ell^2(V,\mu),$$
which makes $\ell^2(V,\mu)$ a Hilbert space.

For any function $f\in C(V)$, the $\mu$-Laplacian $\Delta$ of $f$ is defined by
$$\Delta f(x)=\frac{1}{\mu(x)}\sum_{y\sim x}\omega_{xy}(f(y)-f(x)),$$
it can be checked that $D_\mu<\infty$ is equivalent to the $\mu$-Laplacian $\Delta$ being bounded on $\ell^p(V,\mu)$ for
all $p\in [1,\infty]$ (see \cite{HAESELER}).

The gradient form $\Gamma$ associated with a $\mu$-Laplacian is defined by
$$\Gamma(f,g)(x)=\frac{1}{2\mu(x)}\sum_{y\sim x}\omega_{xy}(f(y)-f(x))(g(y)-g(x)).$$
We write $\Gamma(f)=\Gamma(f,f)$.

The connected graph can be endowed with its natural graph metric $d(x,y)$, i.e. the smallest number of edges of a path
between two vertices $x$ and $y$, then we define balls $B(x,r)=\{y\in V:d(x,y)\leq r\}$ for any $r\geq0$. The volume of a
subset $A$ of $V$ can be written as $V(A)$ and $V(A)=\sum_{x\in A}\mu(x)$, for convenience, we usually abbreviate
$V\big(B(x,r)\big)$ as $V(x,r)$. In addition, a graph $G$ satisfies a polynomial volume growth of degree $m$, if for all
$x\in V$, $r\geq 0$,
\begin{equation*}
V(x,r)\leq cr^m.
\end{equation*}

We say that a function $p:(0,+\infty)\times V \times V\rightarrow \mathbb{R}$ is a fundamental solution of the heat
equation $u_t=\Delta u$ on $G=(V,E)$, if for any bounded initial condition $u_0:V\rightarrow \mathbb{R}$, the function
$$u(t,x)=\sum_{y\in V}p(t,x,y)u_0(y), \quad (t>0, \, x\in V)$$
is differentiable in $t$ and satisfies the heat equation, and for any $x\in V$,
$\lim\limits_{t\rightarrow0^+}u(t,x)=u_0(x)$ holds.

\newtheorem{definition}{\textbf{Definition}}[section]
\begin{definition}
\textnormal{Let a sequence of finite subsets $\{U_i\}_{i=1}^{\infty}$ be an exhaustion of $V$, that is,
$$U_1\subset U_2\subset\cdots\subset U_i\subset\cdots, \quad \textrm{and} \quad \cup_{i=1}^{\infty}U_i=V.$$
If we write $p_k(t,x,y)$ as the heat kernel on $U_k$, then we can define the heat kernel on $G$ by
$$p(t,x,y)=\lim_{k\rightarrow \infty}p_k(t,x,y).$$
This construction was carried out in \cite{MD,AW,RW}. Moreover, the authors showed $p_k\leq p_{k+1}$ for $k\in\mathbb{N}$
and indicated that the definition of the heat kernel is independent of the choice of the exhaustion. }
\end{definition}

For completeness, we recall some important properties of the heat kernel $p(t,x,y)$ (see \cite{HLLY,AW,RW}), as follows
\newtheorem{remark}{\textbf{Remark}}[section]
\begin{remark}
\textnormal{For $t,s>0$ and any $x,y\in V$, we have\\
(i) \quad $p(t,x,y)=p(t,y,x),$\\
(ii) \quad $p(t,x,y)\geq 0$,\\
(iii) \quad $\sum_{y\in V}\mu(y)p(t,x,y)\leq 1$,\\
(iv) \quad $\partial_t p(t,x,y)=\Delta_xp(t,x,y)=\Delta_yp(t,x,y)$,\\
(v) \quad $\sum_{z\in V}\mu(z)p(t,x,z)p(s,z,y)=p(t+s,x,y)$.}
\end{remark}

\newtheorem{lemma}{\textbf{Lemma}}[section]
\begin{lemma}
\textnormal{For all $x\in V$, $p(t,x,x)$ is non-increasing for $t\in (0,\infty)$.}
\end{lemma}

\begin{proof}
By (v) of Remark 2.1, for any $t>0$ and $x\in V$,
$$p(t,x,x)=\sum_{y\in V}\mu(y)p^2\left(t/2,x,y\right).$$

It follows that
\begin{equation*}
\tag{2.1}
\begin{split}
\partial_t p(t,x,x)&=\partial_t \sum_{y\in V}\mu(y)p^2\left(t/2,x,y\right)\\
&=\lim_{k\rightarrow \infty}\partial_t \sum_{y\in V}\mu(y)p_k^2\left(t/2,x,y\right)\\
&=\lim_{k\rightarrow \infty}\sum_{y\in V}\mu(y)\partial_t\left( p_k^2\left(t/2,x,y\right)\right)\\
&=\lim_{k\rightarrow \infty}\sum_{y\in V}\mu(y)p_k\left(t/2,x,y\right)\Delta_x\left( p_k\left(t/2,x,y\right)\right)\\
&=\sum_{y\in V}\mu(y)p\left(t/2,x,y\right)\Delta_x\left( p\left(t/2,x,y\right)\right)\\
&=-\sum_{y\in V}\mu(y)\Gamma(p)(t/2,x,y)\\
&\leq 0.
\end{split}
\end{equation*}

\textbf{Note.} In the deduction process described above:
(i) The interchange of limitation and summation(or the interchange
of deviation and summation) are based on the fact that $p_k(t/2,x,y)$ is non-zero only for finitely many $y$.
(ii) The interchange of limitation and deviation in the second step of the calculation from above is due to
the uniform convergence of the sequences, the details are as follows:

Since $p_k$ is monotonous with respect to $k$, which, together
with the Dini theorem yields that $\{p_k(t/2,x,y)\}$ and
$\{p_k^2(t/2,x,y)\}$ both are uniformly convergent. Furthermore,
according to the definition of $\Delta$ and $D_\mu<\infty$,
$\{\Delta( p_k(t/2,x,y))\}$ also converges uniformly.

In conclusion, we obtain the equalities (2.1).
And the last equality turns out that the heat kernel $p(t,x,x)$
is non-increasing with respect to $t\in (0,\infty)$.

This completes the proof of Lemma 2.1.
\end{proof}

For global pointwise upper bounds of heat kernel $p(t,x,y)$ on general graph, Davies \cite{Davies2} obtained an important
proposition. It states that
\begin{lemma}[see \cite{Davies2}]
\textnormal{Let $\Phi$ be the set of all positive functions $\phi$ on $V$ such that $\phi^{\pm 1}\in \ell^\infty$, for
$x,y\in V$ and all $t>0$, we have
\begin{equation*}
p(t,x,y)\leq \big(\mu(x)\mu(y)\big)^{-\frac{1}{2}}\inf_{\phi\in \Phi}\{\phi(x)^{-1}\phi(y)e^{h(\phi)t}\},
\end{equation*}
where $h(\phi)=\sup_{x\in V}b(\phi,x)-\Lambda$ and
\begin{equation*}
b(\phi,x)=\frac{1}{2\mu(x)}\sum_{y\sim x}\omega_{xy} \left(\frac{\phi(y)}{\phi(x)}+\frac{\phi(x)}{\phi(y)}-2\right).
\end{equation*}
}
\end{lemma}

\section{Main Results}

Our main results are stated in the following theorems.
\newtheorem{theorem}{\textbf{Theorem}}[section]
\begin{theorem}
\textnormal{If $x_1,x_2\in V$, then for all $t> 0$, we have
\begin{equation*}
\tag{3.1} p(t,x_1,x_2)\leq
\big(\mu(x_1)\mu(x_2)\big)^{-\frac{1}{2}}\exp\left(-\frac{d(x_1,x_2)}{2}\log\left(\frac{d(x_1,x_2)}{2D_\mu
et}\right)-\Lambda t \right),
\end{equation*}
where $\Lambda\geq 0$ is the bottom of the $\ell^2$ spectrum of $-\Delta$. }
\end{theorem}

\begin{remark}\textnormal{The upper bound formulated by (3.1) is similar to the Folz's
result in \cite{Folz}, but the metric in our results is natural
graph metric, which is the difference between them.}
\end{remark}

\begin{theorem}
\textnormal{Assume that, for all $x\in V$ and $r\geq r_0$,
\begin{equation*}
\tag{3.2} V(x,r)\leq c_0 r^m,
\end{equation*}
where $r_0,c_0,m$ are some positive constants. Then, for all large
enough $t$,
\begin{equation*}
\tag{3.3} p(t,x,x)\geq \frac{1}{4V(x,C t\log t)},
\end{equation*}
where $C>2D_\mu e$.}
\end{theorem}

\begin{remark}
\textnormal{
Although this estimate may be not as sharp as Gaussian estimate in
\cite{HLLY}, its condition is far weaker than the latter.
Actually, Bauer et al. \cite{BHLLMY} have concluded that the
$CDE(n,0)$ and $CDE'(n,0)$ implies the volume growth (3.2). So, in
some cases, the estimate (3.3) would have broader applications
compared to the Gaussian estimate in \cite{HLLY}.}
\end{remark}

\begin{remark}
\textnormal{
In \cite{Grigoryan2}, Coulhon and Grigoryan proved the pointwise lower bound of $p(t,x,x)$ on non-compact manifolds $M$,
that is, if for all $r\geq r'$ and some $x\in M$,
\begin{equation*}
V(x,r)\leq c'r^m,
\end{equation*}
then for all $t>t_0$,
\begin{equation*}
p(t,x,x)\geq  \frac{1}{4V(x,\sqrt{C't\log t})}.
\end{equation*}
Moreover, in \cite{Grigoryan3}, Grigoryan established the lower bound for discrete-time kernel $p_{2n}(x,x)$ on graphs,
whose form is
\begin{equation*}
p_{2n}(x,x)\geq  \frac{1}{4V(x,\sqrt{C'n\log (2n)})}.
\end{equation*}
However, in the continuous-time setting, the lower bound for heat kernel $p(t,x,x)$
will have a weaker form, i.e., inequality (3.3).
}
\end{remark}

\section{Proof of the main results}
We are now in a position to prove our main results.

\begin{proof}[Proof of Theorem 3.1.]
According to the Lemma 2.2, we have, for all $x,y\in V$ and $t>0$, the estimate
\begin{equation*}
\tag{4.1}
p(t,x,y)\leq \left(\mu(x)\mu(y)\right)^{-\frac{1}{2}}\inf_{\psi\in
\ell^\infty(V,\mu)}\exp\left(\psi(x)-\psi(y)+h(\psi)t\right),
\end{equation*}
where $h(\psi)=\sup_{x\in V}b(\psi,x)-\Lambda$ and
\begin{equation*}
b(\psi,x)=\frac{1}{2\mu(x)}\sum_{y\sim x}\omega_{xy} \left(e^{\psi(y)-\psi(x)}+e^{\psi(x)-\psi(y)}-2\right).
\end{equation*}

Fix $x_1,x_2\in V$ and set $D:=d(x_1,x_2)$. For $s>0$, we define
\begin{equation*}
\psi(x)=s\left(D\wedge d(x,x_1)\right)\in \ell^\infty(V,\mu).
\end{equation*}

Using the triangle inequality for the graph metric $d(x,y)$, we have
\begin{equation*}
\psi(y)-\psi(x)\leq sd(x,y), \quad \, (y\sim x).
\end{equation*}

On the other hand, from the fact that the function $g(t)=e^t+e^{-t}=2\cosh(t)$ is increasing for $t\in (0,+\infty)$, we
obtain
\begin{equation*}
\begin{split}
b(\psi,x)&\leq\frac{1}{2\mu(x)}\sum_{y\sim x}\omega_{xy} \left(e^{s d(x,y)}+e^{-s d(x,y)}-2\right)\\
&=\frac{1}{2\mu(x)}\sum_{y\sim x}\omega_{xy} \left(e^{s}+e^{-s}-2\right).
\end{split}
\end{equation*}

Considering a function
$$f(v)=v+\frac{1}{v}-2-v\log^2 v, \;\;  v\in[1,\infty).$$

Differentiating with respect to $v$ gives
\begin{equation*}
\begin{split}
f'(v) &= -\frac{1}{v^{2}}\left(  v^{2}\log^{2}v+2v^{2}\log
v-v^{2}+1 \right)\\
&= -\frac{1}{v^{2}}f_{1}(v)
\end{split}
\end{equation*}
and
\begin{equation*}
f_{1}'(v)=2v\left(  \log v\right)  \left(  \log v+3\right)
\geq 0, \quad v\in[1,\infty).
\end{equation*}

Hence, we conclude that
$$f_{1}(v)\geq f_{1}(1)=0, \;\;  v\in[1,\infty),$$
and then
$$f(v)\leq f(1)=0, \;\;  v\in[1,\infty),$$
which implies that
$$f(v)=v+\frac{1}{v}-2-v\log^{2}v\leq 0, \;\; v\in[1,\infty).$$

Taking $v=e^s$ in the above equality, then we have
$e^{s}+e^{-s}-2\leq s^2e^s$, which gives
\begin{equation*}
\begin{split}
b(\psi,x)&\leq \frac{s^2e^s}{2\mu(x)}\sum_{y\sim x}\omega_{xy}\\
&\leq \frac{s^2e^s}{2}D_\mu.
\end{split}
\end{equation*}

Since this estimate holds uniformly in $x$, we obtain
$$h(\psi)=\sup_{x\in V}b(\psi,x)-\Lambda\leq \frac{s^2e^s}{2}D_\mu-\Lambda.$$

Combining this estimate with (4.1), we get, for any $s>0$,
\begin{equation*}
\begin{split}
p(t,x_1,x_2)&\leq \big(\mu(x_1)\mu(x_2)\big)^{-\frac{1}{2}}\exp\left(\psi(x_1)-\psi(x_2)+\frac{s^2e^s}{2}D_\mu t-\Lambda t \right)\\
&=\big(\mu(x_1)\mu(x_2)\big)^{-\frac{1}{2}}\exp\left(-s d(x_1,x_2)+\frac{s^2e^s}{2} D_\mu t-\Lambda t \right)\\
&=\big(\mu(x_1)\mu(x_2)\big)^{-\frac{1}{2}}\exp\left(D_\mu t\left(-s \frac{d(x_1,x_2)}{D_\mu
t}+\frac{s^2e^s}{2}\right)-\Lambda t \right).
\end{split}
\end{equation*}

In view of the arbitrariness of $s$, we have
\begin{equation*}
\tag{4.2}
p(t,x_1,x_2)\leq \min_{s>0}\left\{\big(\mu(x_1)\mu(x_2)\big)^{-\frac{1}{2}}\exp\left(D_\mu t\left(-s
\frac{d(x_1,x_2)}{D_\mu t}+\frac{s^2e^s}{2}\right)-\Lambda t \right)\right\}.
\end{equation*}

Setting $f(s)=\frac{s^2e^s}{2}$ and $\gamma=\frac{d(x_1,x_2)}{D_\mu t}$, the estimate (4.2) becomes
\begin{equation*}
\tag{4.3}
p(t,x_1,x_2)\leq \big(\mu(x_1)\mu(x_2)\big)^{-\frac{1}{2}}\exp\left(D_\mu t\widehat{f}(\gamma)-\Lambda t \right),
\end{equation*}
where $\widehat{f}$ is the Legendre transform of $f$, defined by
$$\widehat{f}(\gamma)=\min_{s>0}\left\{-s\gamma+f(s)\right\}.$$

In \cite{Folz}, Folz concluded that
\begin{equation*}
\widehat{f}(\gamma)\leq -\frac{\gamma}{2}\log\left(\frac{\gamma}{2e}\right).
\end{equation*}

Applying this estimate to (4.3) yields
\begin{equation*}
\begin{split}
p(t,x_1,x_2)&\leq \big(\mu(x_1)\mu(x_2)\big)^{-\frac{1}{2}}\exp\left(D_\mu
t\left(-\frac{\gamma}{2}\log\left(\frac{\gamma}{2e}\right)\right)-\Lambda t \right)\\
&=\big(\mu(x_1)\mu(x_2)\big)^{-\frac{1}{2}}\exp\left(-\frac{d(x_1,x_2)}{2}\log\left(\frac{d(x_1,x_2)}{2D_\mu e
t}\right)-\Lambda t \right).
\end{split}
\end{equation*}

The proof of Theorem 3.1 is complete.
\end{proof}

\begin{proof}[Proof of Theorem 3.2.]
By utilizing the properties of heat kernel and the Cauchy-Schwarz inequality, we obtain, for any $r>0$,

\begin{equation*}
\tag{4.4}
\begin{split}
p(2t,x,x)&=\sum_{z\in V}\mu(z)p^2(t,x,z)\\
&\geq \sum_{z\in B(x,r)}\mu(z)p^2(t,x,z)\\
&\geq \frac{1}{V(x,r)}\left(\sum_{z\in B(x,r)}\mu(z)p(t,x,z)\right)^2\\
&= \frac{1}{V(x,r)}\left(1-\sum_{z\in B(x,r)^c}\mu(z)p(t,x,z)\right)^2.
\end{split}
\end{equation*}

Using Lemma 2.1, we obtain from (4.4) that
\begin{equation*}
\tag{4.5}
\begin{split}
p(t,x,x)\geq \frac{1}{V(x,r)}\left(1-\sum_{z\in B(x,r)^c}\mu(z)p(t,x,z)\right)^2.
\end{split}
\end{equation*}

Suppose that we can find $r=r(t)$ so that
\begin{equation*}
\tag{4.6}
\sum_{z\in B(x,r)^c}\mu(z)p(t,x,z)\leq \frac{1}{2},
\end{equation*}
then (4.5) implies
\begin{equation*}
p(t,x,x)\geq \frac{1}{4V(x,r)},
\end{equation*}
which will allow us to obtain the desired result, if $r=Ct\log t$.

Let us now prove (4.6) with $r(t)=Ct\log t$. According to the Theorem 3.1, we get
\begin{equation*}
\begin{split}
p(t,x,z)&\leq \big(\mu(x)\mu(z)\big)^{-\frac{1}{2}}\exp\left(-\frac{d(x,z)}{2}\log\left(\frac{d(x,z)}{2D_\mu
et}\right)-\Lambda t \right)\\
&\leq \frac{1}{\mu_0}\exp\left(-\frac{d(x,z)}{2}\log\left(\frac{d(x,z)}{2D_\mu et}\right)-\Lambda t \right),
\end{split}
\end{equation*}
where $\mu_0:=\inf_{x\in V}\mu(x)>0$.

Hence, for $r\geq r_0$, we have
\begin{equation*}
\begin{split}
\sum_{z\in B(x,r)^c}\mu(z)p(t,x,z)&\leq \frac{1}{\mu_0}\sum_{z\in
B(x,r)^c}\mu(z)\exp\left(-\frac{d(x,z)}{2}\log\left(\frac{d(x,z)}{2D_\mu et}\right)-\Lambda t \right)\\
&=\frac{1}{\mu_0}\sum_{k=0}^{\infty}\sum_{z\in B(x,2^{k+1}r)\backslash B(x,2^k
r)}\mu(z)\exp\left(-\frac{d(x,z)}{2}\log\left(\frac{d(x,z)}{2D_\mu et}\right)-\Lambda t \right)\\
&\leq\frac{1}{\mu_0}\sum_{k=0}^{\infty}\sum_{z\in B(x,2^{k+1}r)\backslash B(x,2^k r)}\mu(z)\exp\left(-\frac{2^k
r}{2}\log\left(\frac{2^kr}{2D_\mu et}\right)-\Lambda t \right)\\
&\leq\frac{1}{\mu_0}\sum_{k=0}^{\infty}\exp\left(-\frac{2^k r}{2}\log\left(\frac{2^kr}{2D_\mu et}\right)-\Lambda t
\right)V(x,2^{k+1}r)\\
&\leq\frac{c_0}{\mu_0}\sum_{k=0}^{\infty}(2^{k+1}r)^m\exp\left(-\frac{2^k r}{2}\log\left(\frac{2^kr}{2D_\mu
et}\right)-\Lambda t \right),\\
\end{split}
\end{equation*}
where we split the complement of $B(x,r)$ into the union of the annuli $B(x,2^{k+1}r)\backslash B(x,2^k r)$,
$k=0,1,2,\ldots$, and use the fact that $\exp\left(-\frac{y}{2}\log \frac{y}{2D_\mu et}-\Lambda t\right)$ is decreasing
with respect to $y$.
Setting
\begin{equation*}
a_k=(2^{k+1}r)^m\exp\left(-\frac{2^k r}{2}\log\left(\frac{2^kr}{2D_\mu et}\right)-\Lambda t \right), \quad \,
(k=0,1,\cdots).
\end{equation*}

Direct calculation gives
\begin{equation*}
\begin{split}
\frac{a_{k+1}}{a_k}&=2^m\exp\left(-2^k r\log\left(\frac{2^kr}{D_\mu et}\right)+2^{k-1} r\log\left(\frac{2^{k-1}r}{D_\mu
et}\right)\right)\\
&=2^m\exp\left(-2^{k-1}r\left(\log\left(\frac{2^kr}{D_\mu et}\right)^2-\log\left(\frac{2^{k-1}r}{D_\mu
et}\right)\right)\right)\\
&=2^m\exp\left(-2^{k-1}r\log\left(\frac{2^{k+1}r}{D_\mu et}\right)\right).
\end{split}
\end{equation*}

If $\frac{2r}{D_\mu et}>1$, then
\begin{equation*}
-2^{k-1}r\log\left(\frac{2^{k+1}r}{D_\mu et}\right)\leq -2^{-1}r\log\left(\frac{2r}{D_\mu et}\right),
\end{equation*}
and therefore
\begin{equation*}
\frac{a_{k+1}}{a_k}\leq 2^m\exp\left(-\frac{r}{2}\log\left(\frac{2r}{D_\mu et}\right)\right).
\end{equation*}

Assuming that $\frac{r}{2}\log\left(\frac{2r}{D_\mu et}\right)\geq m$, then we have
\begin{equation*}
\frac{a_{k+1}}{a_k}\leq \left(\frac{2}{e}\right)^m<1,
\end{equation*}
and the sum of $\{a_k\}$ becomes
\begin{equation*}
\sum_{k=0}^{\infty}a_k\leq \frac{a_0}{1-\frac{2}{e}}.
\end{equation*}

Combining the results discussed above, we obtain
\begin{equation*}
\tag{4.7}
\sum_{z\in B(x,r)^c}\mu(z)p(t,x,z)\leq K r^m\exp\left(-\frac{r}{2}\log\left(\frac{r}{2D_\mu et}\right)-\Lambda t \right)
\end{equation*}
for $r\geq r_0$,\, $\frac{2r}{D_\mu et}>1$ and $\frac{r}{2}\log\left(\frac{2r}{D_\mu et}\right)\geq m$, where $K=\frac{2^mc_0}{\mu_0\left(1-\frac{2}{e}\right)}$.

Choosing $r=r(t)=Ct\log t$, where $C$ is a positive constant satisfying $C>2D_\mu e$. Besides, in any case, the $r(t)$
will satisfy the following conditions:
$$\mathrm{(i)}\;\;r(t)\geq r_0;\;\;\mathrm{(ii)}\;\;\frac{2r(t)}{D_\mu et}>1;\;\;\mathrm{(iii)}\;\;
\frac{r(t)}{2}\log\left(\frac{2r(t)}{D_\mu et}\right)\geq m.$$

By the monotonicity of $r(t)$, the condition (i) listed above is easy to be achieved, here we assume that $t_1$ is the
minimum of $t$ which satisfies the condition (i). Indeed, by combining (i) and (iii), one can deduce the condition (ii)
directly. As for the condition (iii), we observe that the left hand side of (iii) tends to $+\infty$ as
$t\rightarrow +\infty$, so we can choose a real number $t_2$ such that the left hand side of (iii) is equal or larger
than $m$ for all $t\geq t_2$. And then substituting $r(t)=Ct\log t$ into (4.7), it follows that                                                                  \begin{equation*}
\tag{4.8}
\begin{split}
\sum_{z\in B(x,r)^c}\mu(z)p(t,x,z)
&\leq K \left(Ct\log t\right)^m
\exp\left(-\frac{Ct\log t}{2}\log\left(\frac{Ct\log t}{2D_\mu et}\right)-\Lambda t \right)\\
&= KC^me^{-\Lambda t}(\log t)^m
t^{m}\left(\frac{C\log t}{2D_\mu e}\right)^{-\frac{Ct\log t}{2}}\\
&=KC^me^{-\Lambda t}
t^{m}(\log t)^{m-\frac{Ct\log t}{2}}\left(\frac{C}{2D_\mu e}\right)^{-\frac{Ct\log t}{2}}.
\end{split}
\end{equation*}

In view of the above assumption $C>2D_\mu e$, which implies that
$$\left(\frac{C}{2D_\mu e}\right)^{-\frac{Ct\log t}{2}}\rightarrow 0\quad \textrm{as}\quad t\rightarrow +\infty.$$

In addition, it is easy to observe that
$$\lim_{t\rightarrow +\infty}t^m(\log t)^{m-\log t}=0.$$

We thus conclude that the right hand side of (4.8) approaches $0$
as $t\rightarrow +\infty$. So, we can choose a real number $T\geq\max\{t_1,t_2\}$
such that the right hand side of (4.8) is equal or less than $\frac{1}{2}$ for all $t\geq T$.
This yields the required inequality (4.6).

The proof of Theorem 3.2 is complete.
\end{proof}

\section*{Acknowledgments}
This research is supported by the National Science Foundation of China (Grant No.11671401).

\def\refname{\Large

\vspace{25pt}

{\setlength{\parindent}{0pt}
Yong Lin,\\
Department of Mathematics, Renmin University of China,  Beijing, 100872, P. R. China\\
linyong01@ruc.edu.cn\\
Yiting Wu,\\
Department of Mathematics, Renmin University of China,  Beijing, 100872, P. R. China\\
yitingly@126.com
                   }
\end{document}